\documentclass[a4paper]{amsart}
\usepackage{graphicx}
\usepackage{amssymb}
\usepackage{mathrsfs}
\usepackage{nicematrix}
\usepackage{stmaryrd}
\usepackage[mathcal]{euscript}
\usepackage{tikz}
\usepackage{tikz-cd}
\usepackage{hyperref}
\hypersetup{%
  bookmarksnumbered=true,%
  colorlinks=true,%
  linkcolor=blue,%
  citecolor=blue,%
  filecolor=blue,%
  menucolor=blue,%
  urlcolor=blue,%
  bookmarksopen=true,%
  bookmarksdepth=2,%
  pageanchor=true}

\title[Serre's theorem for coherent sheaves]{Serre's theorem for coherent sheaves via Auslander's techniques}

\author{Henning Krause}
\address{Fakult\"at f\"ur Mathematik\\
Universit\"at Bielefeld\\ D-33501 Bielefeld\\ Germany}
\email{hkrause@math.uni-bielefeld.de}

\theoremstyle{plain}
\newtheorem{thm}{Theorem}
\newtheorem{prop}[thm]{Proposition}
\newtheorem{lem}[thm]{Lemma} 

\newtheorem{cor}[thm]{Corollary}

\theoremstyle{definition}

\newtheorem{exm}[thm]{Example}

\theoremstyle{remark}
\newtheorem{rem}[thm]{Remark}

\numberwithin{equation}{section}

\newcommand{\coh}{\operatorname{coh}}

\newcommand{\Coker}{\operatorname{Coker}}

\newcommand{\eff}{\operatorname{eff}}

\newcommand{\End}{\operatorname{End}}

\newcommand{\Ext}{\operatorname{Ext}}

\newcommand{\grmod}{\operatorname{grmod}}
\newcommand{\grproj}{\operatorname{grproj}}
\newcommand{\GrMod}{\operatorname{GrMod}}

\newcommand{\Hom}{\operatorname{Hom}}

\renewcommand{\mod}{\operatorname{mod}}
\newcommand{\Mod}{\operatorname{Mod}}

\newcommand{\rad}{\operatorname{rad}}
\newcommand{\Rad}{\operatorname{Rad}}

\newcommand{\RHom}{\operatorname{\mathbf{R}Hom}}

\newcommand{\Ab}{\mathrm{Ab}}

\newcommand{\op}{\mathrm{op}}

\newcommand{\iso}{\xrightarrow{\raisebox{-.4ex}[0ex][0ex]{$\scriptstyle{\sim}$}}}

\newcommand{\smatrix}[1]{\left[\begin{smallmatrix}#1\end{smallmatrix}\right]}

\newcommand{\xto}{\xrightarrow}
\newcommand*{\intref}[2]{\def\tmp{#1}\ifx\tmp\empty\hyperref[#2]{\ref*{#2}}\else\hyperref[#2]{#1~\ref*{#2}}\fi}

\def\A{\mathcal A} 
 
\def\C{\mathcal C}
\def\D{\mathcal D}

\def\H{\mathcal H}
\def\I{\mathcal I}
\def\K{\mathcal K} 
\def\L{\mathcal L} 

\def\P{\mathcal P}
\def\R{\mathcal R}

\def\bfD{\mathbf D}

\def\bbP{\mathbb P}

\def\bbX{\mathbb X} 
 
\def\bbZ{\mathbb Z}

\def\d{\delta}
\def\g{\gamma}
\def\p{\phi}
\def\s{\sigma}
\def\t{\tau}

\def\Ga{\Gamma}
\def\La{\Lambda}

\begin{document}



\date{\today}

\maketitle

\begin{abstract}
  For an abelian category and a distinguished object with a graded
  endomorphism ring a necessary and sufficient criterion is given so
  that the category is equivalent to the abelian quotient of the
  category of finitely presented graded modules modulo the Serre
  subcategory of finite length modules. A particular example is the
  category of coherent sheaves on a projective variety, following a
  theorem of Serre from 1955. The proof uses Auslander's theory of
  coherent functors, and there are no noetherianess assumptions. A
  theorem of Lenzing for representations of hereditary algebras is
  given as an application.
\end{abstract}

\section{Introduction}

Serre's theorem relating coherent sheaves on a projective variety to
graded modules over some appropriate graded ring is a cornerstone of
non-commutative geometry, for instance by the non-commutative analogue
of this theorem due to Artin and Zhang \cite{AZ1994,Se1955}. In this
note we revisit the theory, motivated by work of Lenzing which
introduces for any finite dimensional hereditary algebra of infinite
representation type a category of coherent sheaves \cite{Le1986}. The
corresponding graded ring is the preprojective algebra \cite{BGL1987},
which appeared already in work of Gel'fand and Ponomarev \cite{GP1979}
as well as in work of Dlab and Ringel \cite{DR1980}. The novel aspect
of the present work is a new formulation and proof of the analogue of
Serre's theorem which is based on Auslander's theory of coherent
functors \cite{Au1966}. In particular, we do not impose any
noetherianess assumptions. This is necessary because the preprojective
algebra is noetherian only when the corresponding hereditary algebra
is of tame representation type \cite{Ba1985,BGL1987}.

\section{An analogue of Serre's theorem}

We formulate the analogue of Serre's theorem in the setting of
Krull--Schmidt categories. Let $\C$ be an additive category that is
\emph{Krull--Schmidt}. Thus any object $X$ admits a finite
decomposition $X=\bigoplus_i X_i$ such that each $X_i$ is
indecomposable with a local endomorphism ring. The \emph{radical}
$\Rad\C$ of $\C$ is the ideal given by subgroups
\[\Rad(X,Y)\subseteq\Hom(X,Y)\qquad(X,Y\in\C)\]
consisting of all morphisms $X\to Y$ such that for any pair of
decompositions $X=\bigoplus_i X_i$ and $Y=\bigoplus_j Y_j$ into
indecomposable objects no component $X_i\to Y_j$ is an
isomorphism. Let $\Rad^0\C$ denote the ideal of all morphisms in $\C$
and for $n\ge 0$ we set $\Rad^{n+1}\C=(\Rad\C)(\Rad^n\C)$.  We say
that the \emph{length} of a morphism $X\to Y$ is bounded by $n$ if all
morphisms $X'\to Y$ from $\Rad^n\C$ factor through $X\to Y$.

For a $\bbZ$-graded ring $A=\bigoplus_{n\in\bbZ}A_n$ we consider
the category $\GrMod A$ of $\bbZ$-graded right $A$-modules. Let $\grmod A$ denote the
full subcategory of finitely presented modules and $\grproj A$ denotes
the full subcategory of finitely generated projective modules. We
write $\grmod_0A$ for the full subcategory of all finite length
modules and always assume $\grmod_0A\subseteq\grmod A$; this is
automatic if the ring is right noetherian.

We say that a graded ring $A$ is \emph{right coherent} if $\grmod A$
is an abelian category, and $A$ is \emph{semiperfect} if
$\grproj A$ is Krull--Schmidt. These properties hold, for instance,
when $A$ is a right noetherian algebra over a field such that each
homogeneous component $A_n$ is finite dimensional; see \cite{At1956,Kr2015}
for the Krull--Schmidt property.

\begin{thm}\label{th:serre}    
  Let $(\A,C,\s)$ be a triple consisting of an abelian category $\A$,
  a distinguished object $C$, and an equivalence
  $\s\colon\A\iso\A$. Suppose that $\Hom(C,\s^nC)=0$ for all $n <0$
  and that the graded ring
\[A=\bigoplus_{n\ge 0}\Hom(C,\s^nC)\] is right coherent
and semiperfect with $\grmod_0A\subseteq\grmod A$. Then the assignment
\[X\mapsto \Ga_*(X):=\bigoplus_{n\in\bbZ}\Hom(C,\s^nX)\qquad
  (X\in\A)\] admits a partial left adjoint $T\colon\grmod A\to\A$,
which is right exact and determined by $T(A)=C$. The functor $T$ is
exact and induces an equivalence \[(\grmod A)/(\grmod_0 A)\iso\A\] if
and only if the full subcategory $\C\subseteq\A$ consisting of all
direct summands of finite direct sums of objects $\s^{n}C$ with
$n\in\bbZ$ satisfies the following:
\begin{enumerate}
\item[(A1)]  Every object $X\in\A$ admits an epimorphism $C\to X$ with
  $C\in\C$, which can be chosen in $\Rad\C$ when $X\in\C$.
\item[(A2)]  A morphism in $\C$ has finite length if it is an
  epimorphism in $\A$.
\end{enumerate}
\end{thm}

Fix a triple $(\A,C,\s)$ as above. The algebra
$A=\bigoplus_{n\ge 0}A_n$ is the \emph{orbit algebra} of $C$ with
multiplication given by $xy=(\s^q x)\circ y$ for $x\in A_p$ and
$y\in A_q$.  The assignment $X\mapsto\Ga_*(X)$ yields an equivalence
$\C\iso\grproj A$ and $T\colon\grmod A\to\A$ is the right exact
functor extending
\[\grproj A\iso\C\hookrightarrow\A.\]
Thus we have the following
diagram \[\begin{tikzcd}
\grmod A\arrow[swap]{rd}{T}\arrow[hook]{rr}&&\GrMod A\\
&\A\arrow[swap]{ru}{\Ga_*}& 
\end{tikzcd}\]
and $\Ga_*$ corestricts to a right adjoint of $T$  if and only if
$\C\subseteq\A$ is contravariantly finite so that $\Ga_*(X)$ is a
finitely presented $A$-module for each $X\in\A$; see
Remark~\ref{re:adjoint}. Nonetheless, we have for all $X\in\grmod A$
and $Y\in\A$ the adjointness
isomorphism
\[\Hom_\A(T(X),Y)\cong\Hom_A(X,\Ga_*(Y)).\]

The conditions (A1)--(A2) express the `ampleness' of the pair
$(C,\s)$.  The prototypical example is a theorem of Serre
\cite{Se1955} and we refer to \cite{AZ1994,Po2005} for non-commuative
analogues.

\begin{exm}[Serre]\label{ex:serre}
  Consider the projective variety $\bbP^r(K)$ given by an
  algebraically closed field $K$ and an integer $r\ge 1$. Let
  $\coh\bbP^r(K)$ denote the category of coherent sheaves on
  $\bbP^r(K)$.  The assignment
  \[\mathscr F\mapsto \Ga_*(\mathscr
    F)=\bigoplus_{n\in\bbZ}\Hom(\mathscr O, \mathscr F(n))\] yields a
  functor $\coh\bbP^r(K)\to\grmod A$ to the category of graded modules
  over the orbit algebra of the structure sheaf
  \[A=\bigoplus_{n\ge 0}\Hom(\mathscr O, \mathscr O(n))\cong
    K[t_0,\ldots,t_r].\] Taking a module $M$ to its associated sheaf
  $\widetilde M$ provides a left ajoint functor, which is exact and
  annihilates all modules of finite length. In
  \cite[N\textsuperscript{o}~65]{Se1955} it is shown that the
  canonical morphism $\widetilde{\Ga_*(\mathscr F)} \to \mathscr F$ is
  an isomorphism, while kernel and cokernel of the canonical morphism
  $M\to \Ga_*(\widetilde M)$ are of finite length. An equivalent
  statement is that $M\mapsto\widetilde M$ induces an
  equivalence \[(\grmod A)/(\grmod_0 A)\iso \coh\bbX.\]
\end{exm}

Let us comment on the notions `finite length' and
`torsion' for graded modules, because it is common to work
modulo the category of torsion modules in the context of Serre's theorem.

\begin{rem}
  Let $A=\bigoplus_{n\ge 0}A_n$ be a graded ring and suppose that the ring
  $A_0=A/A_{\ge 1}$ is semisimple. Then a graded $A$-module $M$ has finite
  length if and only if $M$ is finitely generated and torsion, so $M_n=0$ for
  $n\gg 0$. To see this fix a  finitely
  generated module $M$ and consider the sequence
  \[\cdots\twoheadrightarrow M/M_{\ge 2}\twoheadrightarrow M/M_{\ge
      1}\twoheadrightarrow M/M_{\ge 0}\twoheadrightarrow M/M_{\ge
      -1}\twoheadrightarrow\cdots\] which stabilises in the right
  direction since $M$ is finitely generated. It stabilises in both
  directions if and only if $M$ has finite length, since each
  subquotient has finite length (using the assumption
  $\grmod_0A\subseteq\grmod A$ so that  each
  subquotient is finitely generated over $A_0$).
\end{rem}

\section{A relative Auslander formula}

The proof of the analogue of Serre's theorem is based on a relative
version of Auslander's formula; it is established in this section
and may be of independent interest. For instance, the
Gabriel--Popescu theorem for Grothendieck categories is another
result that can be explained in terms  of this Auslander formula.

Let $\C$ be an additive category.  Following Auslander \cite{Au1966}
an additive functor $F\colon\C^\op\to\Ab$ into the category of abelian
groups is called \emph{finitely presented}
or \emph{coherent} if it admits a presentation
\begin{equation}\label{eq:F}
    \Hom(-,X)\to \Hom(-,Y)\to F\to 0.
  \end{equation}
  In that case $X\to Y$ is called the \emph{presenting morphism}.
  Let $\mod\C$ denote the category of finitely presented functors
  $\C^\op\to\Ab$. The assignment $X\mapsto \Hom(-,X)$ yields the fully
  faithful \emph{Yoneda functor} $\C\to\mod\C$.  We collect some basic
  facts from \cite{Au1966} which will be used throughout without
  further reference. These properties reflect the fact that
  $\C\to\mod\C$ is nothing but the completion of $\C$ under finite
  colimits.

A morphism $X\to Y$ is a \emph{weak kernel}
of a morphism $Y\to Z$  if the induced sequence
$\Hom(C,X)\to \Hom(C,Y)\to \Hom(C,Z)$ is exact for all objects $C$.

\begin{lem}
  For an additive category $\C$ we have the following.
\begin{enumerate}
\item The category $\mod\C$ is additive and every morphism has a
  cokernel; it is abelian iff every morphism
  in $\C$ has a weak kernel.
\item The  Yoneda functor $\C\to\mod\C$ admits an exact left
  adjoint if $\C$ is abelian.
\item An additive functor  $\C\to\A$ into a category with
  cokernels extends  to a right exact functor $\mod\C\to\A$.
\item An  additive functor
  $f\colon\C\to\D$ extends to a right exact functor
  $f^*\colon\mod\C\to\mod\D$, and $f^*$ is fully faithful if and only
  if $f$ is fully faithful.
\end{enumerate}
\end{lem}
\begin{proof}
  See \cite{Au1966} or Sections~2.1--2.3 in \cite{Kr2021}.
\end{proof}

\begin{exm}
  For a graded ring $A$ the inclusion $\grproj A\to\grmod A$ induces
  an equivalence  $\mod(\grproj A)\iso\grmod A$.
\end{exm}

Let $\A$ be an abelian category. A functor $F\in\mod\A$ with
presentation \eqref{eq:F} is called \emph{effaceable} if it belongs to
the kernel of the left adjoint of the Yoneda functor $\A\to\mod\A$. An
equivalent condition is that the presenting morphism
$\p\colon X\to Y$ is an epimorphism, because the left adjoint
preserves cokernels, so sends
$F=\Coker\Hom(-,\p)$ to $\Coker\p$.  We write $\eff\A$ for the full
subcategory of effaceable functors and note that it is a Serre
subcategory of $\mod\A$.

\begin{prop}[Auslander]
  \pushQED{\qed}
  The right exact functor $\mod\A\to\A$ extending the identity 
  $\A\to\A$ is exact and induces an equivalence
\[(\mod\A)/(\eff\A)\iso\A.\qedhere\]
\end{prop}

This result from \cite{Au1966} is also known as \emph{Auslander's
  formula} \cite{Le1998}. We need the following relative version.
Let  $i\colon\C\to\A$ denote the inclusion of
a full additive subcategory.  We view the induced functor
$i^*\colon\mod\C\to\mod\A$ as an inclusion and set
\[\eff(\A,\C):=\mod\C\cap\eff\A.\]
The subcategory $\C$ \emph{generates} $\A$ if every object $X\in\A$ admits an
epimorphism $C\to X$ with $C\in\C$.

\begin{prop}\label{pr:auslander}
  Let $\A$ be an abelian category and $\C\subseteq\A$ a full additive
  subcategory such that $\mod\C$ is abelian. Then $\C$ generates
  $\A$ if and only if the right exact functor $\mod\C\to\A$ extending
  the inclusion $\C\to\A$ is exact and induces an equivalence
 \[(\mod\C)/(\eff(\A,\C))\iso\A.\] 
\end{prop}

\begin{proof}
  We write the right exact functor $\mod\C\to\A$ as composite
  \[\mod\C\xto{i^*}\mod\A\twoheadrightarrow\A\]
  with $i^*$ induced by the inclusion $i\colon\C\to\A$. The functor $i^*$ is fully
  faithful, and  it is exact when any weak kernel sequence in $\C$ is exact
  in $\A$. The latter property follows when $\C$ generates
  $\A$. Moreover, the condition that $\C$ generates is equivalent to
  the property of  $\mod\C\to\A$  to be essentially surjective.

  Now suppose that $\C$ generates $\A$. Then the exact functor $i^*$
  induces a functor
\[(\mod\C)/(\eff(\A,\C))\to (\mod\A)/(\eff\A)\] by the definition of
$\eff(\A,\C)$. The morphisms in the
localised categories are computed via a calculus of fractions. From
this it follows that the functor is fully faithful when the following
cofinality condition holds: Every $F\in\mod\A$ admits a morphism
$F'\to F$ such that $F'\in\mod\C$ and the image under
$\mod \A\twoheadrightarrow\A$ is an isomorphism; cf.\
\cite[Lemma~1.2.5]{Kr2021}. For this cofinality see Lemma~\ref{le:cofinal} below,
and it remains to compose this functor with the
equivalence $(\mod\A)/(\eff\A)\iso\A$.
\end{proof}

\begin{lem}\label{le:cofinal}
  Let $\A$ be an abelian category and $\C\subseteq\A$ a full additive
  subcategory such that $\mod\C$ is abelian and $\C$
  generates $\A$. Then every $F\in\mod\A$ admits a morphism $F'\to F$
  such that $F'\in\mod\C$ and the image under
  $\mod \A\twoheadrightarrow\A$ is an isomorphism.
\end{lem}
\begin{proof}
Let $X\to Y$ be the morphism presenting $F$. We find objects
$C_i,D_i$ in $\C$ and morphisms such that the following diagram in $\A$
commutes and has exact rows.
 \[\begin{tikzcd}
C_1\arrow{r}{\g}\arrow{d}&C_0\arrow{r}\arrow{d}&X\arrow{r}\arrow{d}&0 \\
D_1\arrow{r}{\d}&D_0\arrow{r}&Y\arrow{r}&0 
\end{tikzcd}\]
This induces the following commutative diagram with
exact rows in $\mod\A$.
 \[\begin{tikzcd}
\Coker\Hom(-,\g)\arrow{r}\arrow{d}&\Coker\Hom(-,\d)\arrow{r}\arrow{d}&F'\arrow{r}\arrow{d}&0 \\
\Hom(-,X)\arrow{r}&\Hom(-,Y)\arrow{r}&F\arrow{r}&0 
\end{tikzcd}\] The top row lies in $\mod\C$. By construction the
vertical morphisms on the left and in the middle are mapped to
isomorphisms under $\mod \A\twoheadrightarrow\A$. It follows that
$F'\to F$ has the desired properties.
\end{proof}

\begin{rem}
  One may think of the relative Auslander formula as a variation of
  the Popescu--Gabriel theorem \cite{PG1964} which says for a
  Grothendieck category $\A$ and a generator $G$ that the left adjoint
  of $\Hom(G,-)\colon\A\to\Mod A$ with $A=\End(G)$ induces an
  equivalence \[(\Mod A)/\L\iso\A\] for some appropriate localising
  subcategory $\L\subseteq\Mod A$. In fact, we may take for
  $\C\subseteq\A$ the full subcategory all coproducts of copies of
  $G$. Then the colimit preserving composite
  $\Mod A\to \mod\C\to\A$ that identifies $A$ with $G$ induces equivalences
\[(\Mod A)/\L\iso (\mod\C)/(\eff(\A,\C))\iso\A.\]
\end{rem}

\begin{rem}\label{re:adjoint}
For a full additive subcategory $\C\subseteq\A$ the right exact functor $\mod\C\to\A$
admits a right adjoint if and only if the subcategory $\C$ is
\emph{contravariantly finite}, so each object $X\in\A$ admits a
morphism $\pi\colon C\to X$ such that $C\in\C$ and each morphism
$C'\to X$ with $C'\in\C$ factors through $\pi$. This condition means
that $\Hom(-,X)|_\C$ belongs to $\mod\C$ for all $X\in\A$, so
$X\mapsto \Hom(-,X)|_\C$ provides the right adjoint.
\end{rem}

\section{Serre's theorem via Auslander's techniques}

We use the relative Auslander formula from
Proposition~\ref{pr:auslander} to prove Theorem~\ref{th:serre}. We
need some preparations and fix an additive category $\C$ that is Krull--Schmidt.  Let
$\mod_0\C$ denote the category of all additive functors $\C^\op\to\Ab$
that are of finite length. For an object $X\in\C$
set \[S_X=\Hom(-,X)/\Rad(-,X).\] We note that $S_X$ is simple when
$\End(X)$ is local, and $S_X\cong\bigoplus_iS_{X_i}$ for any finite
decomposition $X=\bigoplus_i X_i$.

Given an indecomposable object $X$, a radical morphism $X'\to X$ is
called \emph{right almost split} if all radical morphisms terminating
at $X$ factor through $X'\to X$. Simple functors and their connection
to almost split morphisms are discussed in great detail in
\cite[Chapter~II]{Au1976}. In our context the following is needed.

\begin{lem}\label{le:almost-split}
  Let $\C$ be a Krull--Schmidt category. Then
  $\mod_0\C\subseteq\mod\C$ holds if and only if every indecomposable
  object in $\C$ admits a right almost split morphism.
\end{lem}
\begin{proof}
  A functor $F\colon\C^\op\to\Ab$ is simple if and only if
  $F\cong S_X$ for some indecomposable object $X\in\C$; see
  \cite[Proposition~II.1.8]{Au1976}.  Here one uses that $\C$ is
  Krull--Schmidt.  On the other hand, for an indecomposable object
  $X\in\C$ there exists a right almost split morphism $X'\to X$ if
  and only if $S_X$ belongs to $\mod\C$, because a right almost split morphism
  $X'\to X$ amounts to an epimorphism $\Hom(-,X')\to\Rad(-,X)$.
\end{proof}

\begin{lem}\label{le:semisimple}
  Let $\C$ be a Krull--Schmidt category such that $\mod\C$ is
  abelian and $\mod_0\C\subseteq\mod\C$. A functor $F\in\mod\C$
  with presentation \eqref{eq:F}
  belongs to $\mod_0\C$ if and only if the
  morphism $X\to Y$ in $\C$ has finite length.
\end{lem}
\begin{proof}
  We write $\rad F$ for the intersection of all maximal subobjects of
  $F$ and set $\rad^{n+1}F=\rad(\rad^n F)$ for all $n\ge 0$.  For
  $H_X=\Hom(-,X)$ observe that $\rad^n H_X=\Rad^n(-,X)$ for all
  $n\ge 0$; see \cite[Proposition~II.1.8]{Au1976}. Thus the presentation of $F$ induces an epimorphism
  $S_Y\to F/(\rad F)$.  Then the radical filtration
\[\cdots\subseteq\rad^2 F\subseteq\rad^1 F\subseteq\rad^0 F=F\]
lies in $\mod\C$ and each subquotient \[(\rad^n F)/(\rad^{n+1}F)\] has
finite length. It follows that $F$ has finite length if and only if
$\rad^n F=0$ for $n\gg 0$. The presenting morphism $X\to Y$ has by definition
length at most $n$ if and only if the epimorphism $H_Y\to F$ factors
through $H_Y\to H_Y/(\rad^n H_Y)$. On the other hand,  $H_Y\to F$ maps
$\rad^n H_Y$ onto $\rad^n F$, and therefore it factors
through $H_Y\to H_Y/(\rad^n H_Y)$ if and only
if $\rad^n F=0$.
\end{proof}

\begin{lem}\label{le:serre}
  Let $\A$ be an abelian category and $\C\subseteq\A$ a full additive
  subcategory.  Suppose that $\mod\C$ is abelian and that
  $\mod_0\C\subseteq\mod\C$. Then the right exact functor
  $\mod\C\to\A$ extending the inclusion $\C\to\A$ is exact and induces
  an equivalence
  \[(\mod\C)/(\mod_0\C)\iso\A\] if and only if $\C$ generates $\A$ and $\mod_0\C=\eff(\A,\C)$.
\end{lem}
\begin{proof}
The assertion is an immediate consequence of Proposition~\ref{pr:auslander}.
\end{proof}

We are now ready to prove Theorem~\ref{th:serre}. 

\begin{proof}[Proof of Theorem~\ref{th:serre}]
  We recall the equivalences
  \[\C\iso\grproj A\qquad\text{and}\qquad\mod\C\iso\grmod A.\] Our
  assumption $\grmod_0 A\subseteq\grmod A$ implies that every
  indecomposable object in $\C$ admits a right almost split morphism;
  see Lemma~\ref{le:almost-split}.

  Now we apply Lemma~\ref{le:serre} and need to show that conditions
  (A1)--(A2) in Theorem~\ref{th:serre} are equivalent to the
  conditions in Lemma~\ref{le:serre}. This is clear for the first pair
  of conditions expressing the fact that $\C$ is generating $\A$.  It
  follows from Lemma~\ref{le:semisimple} that (A2) holds if and only
  if
  \[\eff(\A,\C)=\mod\C\cap\eff\A\subseteq\mod_0\C.\]
  For the other inclusion it suffices that all simple objects in
  $\mod\C$ belong to $\eff(\A,\C)$. This means that for every
  indecomposable $X\in\C$
  the right almost split morphism $X'\to X$ is an
  epimorphism. But this is precisely the extra condition in (A1) that
  there is an epimorphism $X'\to X$ in $\Rad\C$.
\end{proof}

\section{Hereditary algebras}

Let $K$ be a field and $\La$ a hereditary finite dimensional
$K$-algebra. This means $\Ext_\La^n(-,-)=0$ for all $n>1$. We consider
the category $\mod\La$ of finitely presented $\La$-modules. This
category admits a canonical decomposition. For an additive category
$\C$ we use the notation $\C=\bigvee_i\C_i$ and call this a
\emph{decomposition} when each $\C_i$ is a full additive subcategory
such that each object in $\C$ can be written as a coproduct
$\coprod_i C_i$ with $C_i\in\C_i$ for all $i$, and
$\C_i\cap\big( \bigvee_{j\neq i}\C_j\big)=0$ for all $i$.

Assume that $\La$ is connected and of infinite
representation type. Then there is a decomposition
\[\mod \La=\P\vee\R\vee\I \]
where $\P$ denotes the full subcategory of \emph{preprojective}
$\La$-modules, $\R$ denotes the full subcategory of \emph{regular}
$\La$-modules, and $\I$ denotes the full subcategory of
\emph{postinjective} $\La$-modules \cite[VIII]{ARS1995}. Because
$\A=\mod\La$ is a hereditary abelian category we have a decomposition
of the bounded derived category
\[\bfD^{\mathrm b}(\A)=\bigvee_{n\in\bbZ}\A[n]\]
where $\A[n]$ denotes the full subcategory of complexes with cohomology
concentrated in degree $-n$. Note that this category has Serre duality
which extends the Auslander--Reiten duality for $\A$. Thus there is a
functor $\t\colon \bfD^{\mathrm b}(\A)\iso \bfD^{\mathrm b}(\A)$ such that for all objects
$X,Y$ there is a natural isomorphism
\[D\Hom(X,Y[1])\cong\Hom(Y,\t X)\]
where $D=\Hom_K(-,K)$; see \cite{Kr2021} for details.

In \cite{Le1986} Lenzing proposes a geometric approach and introduces
the following full additive subcategory
\[\H:=\I [-1]\vee\P[0]\vee\R[0]\subseteq\bfD^{\mathrm b}(\mod\La).\]
He shows that $\H$ is a hereditary abelian category with Serre duality
given by the Auslander--Reiten translate $\t\colon\H\iso\H$. One way
of seeing this is that the torsion pair $(\I,\P\vee\R)$ for $\A$
induces a t-structure on $\bfD^{\mathrm b}(\A)$ such that $\H$ identifies with its
heart \cite{HRS1996}. Moreover, $\La\in\P[0]$ is a tilting object and $\RHom(\La,-)$
provides a triangle equivalence \[\bfD^{\mathrm b}(\H)\iso\bfD^{\mathrm b}(\mod\La).\] The
\emph{preprojective algebra} of $\La$ is the orbit algebra
\[\Pi:=\bigoplus_{n\ge 0}\Hom(\La,\t^{-n}\La)\] and the following
theorem is implicit in Lenzing's geometric treatment of hereditary
algebras; cf.\ Theorem~4.10 in \cite{Le1986}. We deduce this from
Theorem~\ref{th:serre}. Note that $\Pi$ is noetherian if and only if
$\La$ is of tame type \cite{Ba1985,BGL1987}.

\begin{thm}\label{th:lenzing}
For a connected hereditary algebra $\La$ of infinite representation type  the
  assignment \[X\mapsto\bigoplus_{n\in\bbZ}\Hom(\La,\t^{-n}X)\]
  induces an equivalence
 \[\H\iso(\grmod\Pi)/(\grmod_0\Pi).\] 
\end{thm}
\begin{proof}
  We set $\s=\t^-$ and the distinguished object is $C=\La$. We need to
  check the conditions in Theorem~\ref{th:serre} for the triple
  $(\H,C,\s)$ and note that $\C=\I [-1]\vee\P[0]$. The preprojective
  algebra $\Pi$ is semiperfect since each homogeneous component is
  finite dimensional over a field. The algebra is right coherent
  because the category $\C$ has kernels. Here one uses that
  $\mod\C\iso\grmod\Pi$.  Serre duality for $\H$ provides for each
  indecomposable $X\in\C$ an almost split sequence
  $0\to \t X\to X'\to X\to 0$ \cite[Theorem~I.3.3]{RV2002}. By the
  definition of an almost split sequence, the morphism $X'\to X$ is
  right almost split. The condition (A1) is clear because we have for
  each regular module $X\in\R[0]$ an epimorphism $X'\to X$ from a
  projective module $X'\in\P[0]$. For an indecomposable object
  $X\in\C$ one uses the right hand morphism $X'\to X$ from the
  corresponding almost split sequence.  To check (A2) consider an
  epimorphism $Y\to Z$ in $\C$ which yields an exact sequence
  $0\to X\to Y\to Z\to 0$. We may apply a power of $\t$ and assume
  that the sequence lies in $\I[-1]$. Consider the induced exact
  sequence
  \[0\to\Hom(-,X)\to\Hom(-,Y)\to\Hom(-,Z)\to\Ext^1(-,X)\] in $\mod\C$.
  We need to show that all morphisms $Z'\to Z$ from $\Rad^n\C$ factor
  through $Y\to Z$ for $n\gg 0$.  For this it suffices to show that
  $\Rad^n\C$ annihilates the functor $\Ext^1(-,X)|_{\I[-1]}$ for
  $n\gg 0$, because $\Hom(P,Z)=0$ for all $P\in\P[0]$.  We
  have \[D\Hom(\t^- X,-)\cong\Ext^1(-,X)\] and from this the claim
  follows, because for any postinjective $\La$-module $M$ we have
  $\Rad^n(M,-)=0$ for $n\gg 0$ in $\mod\La$.
\end{proof}

When passing to derived categories we have the following immediate
consequence; see also \cite[Corollary~5.4]{Mi2012}.

\begin{cor}
  \pushQED{\qed} The quotient functor $\grmod\Pi\twoheadrightarrow\H$
  induces a triangle equivalence
  \[\bfD^{\mathrm b}((\grmod\Pi)/(\grmod_0\Pi))
    \iso\bfD^{\mathrm b}(\H)\iso\bfD^{\mathrm b}(\mod\La).\qedhere\]
\end{cor}

The following example connects the results of Serre and Lenzing; see \cite[Proposition~6.3]{Le1986}.

\begin{exm}
  Consider the Kronecker algebra $\La=\smatrix{K&0\\K^2&K}$. In that
  case $\H$ identifies with the category of coherent shaves on the
  projective line $\bbP^1$. More precisely, the category $\coh\bbP^1$
  admits a tilting object $T=\mathscr
  O_{\bbP^1}\oplus\mathscr O_{\bbP^1}(1)$ such that $\End(T)\cong\La$.
  Then the functor $\RHom(T,-)$ yields a triangle equivalence
  \[\bfD^{\mathrm b}(\coh\bbP^1)\iso\bfD^{\mathrm b}(\mod\La)\]
which   restricts to an equivalence $\coh\bbP^1\iso\H$. Note that this
  identifies the twist $\mathscr F\mapsto \mathscr F(2)$ with
  $X\mapsto \s X$. We have
  \[R:=K[x,y]\cong\bigoplus_{n\ge 0}\Hom(\mathscr O_{\bbP^1},\mathscr
    O_{\bbP^1}(n))\]
  and therefore
  \[\Pi\cong\bigoplus_{n\ge 0}\smatrix{R_{2n}&R_{2n-1}\\R_{2n+1}&R_{2n}}.\]
  The assignment
  \[\bigoplus_{n\in\bbZ}M_n\mapsto
    \bigoplus_{n\in\bbZ}(M_{2n}\oplus M_{2n-1})\] provides an
  equivalence $\grmod R\iso\grmod\Pi$ which identifies the equivalence
  from Serre's theorem in Example~\ref{ex:serre} with the equivalence
  in Theorem~\ref{th:lenzing}, though their twisting objects are
  different ($\mathscr O_{\bbP^1}$ in $\coh\bbP^1$ versus $\La$ in $\H$).
\end{exm}

The example suggests that there are other and arguably better choices
for a distinguished object $C\in\H$ and a twist $\s\colon\H\iso\H$
(different from $C=\La$ and $\s=\t^-$), in particular when $\La$ is a
tame hereditary algebra; see \cite{Ku2008,Ku2009} for a detailed discussion.

\subsection*{Acknowledgements}

I am grateful to Andrew Hubery and Dirk Kussin for inspiration and
helpful comments on this work. Additional thanks to an anonymous
referee for several comments concerning the exposition. This work was
supported by the Deutsche Forschungsgemeinschaft (SFB-TRR 358/1 2023 -
491392403).

\end{document}